\documentclass[12pt]{amsart}

\usepackage{graphicx}
\usepackage{amsmath}
\usepackage{amscd}
\usepackage{amsfonts}
\usepackage{amssymb}
\usepackage{color}
\usepackage{fullpage}
\usepackage{mathtools}
\usepackage[pagebackref,hypertexnames=false, colorlinks, citecolor=red,linkcolor=blue, urlcolor=red]{hyperref}

\usepackage{marginnote}

\numberwithin{equation}{section}

\newtheorem{theorem}{Theorem}[section]
\newtheorem{lemma}[theorem]{Lemma}
\newtheorem{proposition}[theorem]{Proposition}

\newtheorem{corollary}[theorem]{Corollary}

\theoremstyle{definition}
\newtheorem{example}[theorem]{Example}
\newtheorem{remark}[theorem]{Remark}
\newtheorem{definition}[theorem]{Definition}
\newtheorem*{ack}{Acknowledgements}

\newcommand{\be}{\begin{equation}}
\newcommand{\ee}{\end{equation}}
\newcommand{\bes}{\begin{equation*}}
\newcommand{\ees}{\end{equation*}}

\newcommand{\cG}{\mathcal{G}}
\newcommand{\cH}{\mathcal{H}}

\newcommand{\cM}{\mathcal{M}}

\newcommand{\cU}{\mathcal{U}}

\newcommand{\bB}{\mathbb{B}}
\newcommand{\bC}{\mathbb{C}}
\newcommand{\bD}{\mathbb{D}}

\newcommand{\bN}{\mathbb{N}}

\newcommand{\bR}{\mathbb{R}}

\newcommand{\Aut}{\operatorname{Aut}}

\newcommand{\Mult}{\operatorname{Mult}}
\newcommand{\mlt}{\operatorname{Mult}}

\newcommand{\spn}{\operatorname{span}}

\newcommand{\fM}{{\mathfrak{M}}}

\begin{document}

\title{Distance between reproducing kernel Hilbert spaces and geometry of finite sets in the unit ball}

\author{Danny Ofek}
\address{School of Mathematical Sciences\\
Tel Aviv University\\
Tel Aviv\; 6997801\\
Israel}
\email{Dannyofek@mail.tau.ac.il}

\author{Satish K. Pandey}
\address{Faculty of Mathematics\\
Technion - Israel Institute of Technology\\
Haifa\; 3200003\\
Israel}
\email{satishpandey@campus.technion.ac.il}
\urladdr{\href{http://noncommutative.space/}
{\url{http://noncommutative.space/}}}


\author{Orr Moshe Shalit}
\address{Faculty of Mathematics\\
Technion - Israel Institute of Technology\\
Haifa\; 3200003\\
Israel}
\email{oshalit@technion.ac.il}
\urladdr{\href{https://oshalit.net.technion.ac.il/}
{\url{https://oshalit.net.technion.ac.il/}}}

\thanks{The work of Satish K. Pandey is supported in part at the Technion by a fellowship of the Israel Council for Higher Education. The work of O.M. Shalit is partially supported by ISF Grant no. 195/16}
\begin{abstract} 
In this paper we study the relationships between a reproducing kernel Hilbert space, its multiplier algebra, and the geometry of the point set on which they live. 
We introduce a variant of the Banach-Mazur distance suited for measuring the distance between reproducing kernel Hilbert spaces, that quantifies how far two spaces are from being isometrically isomorphic as reproducing kernel Hilbert spaces. 
We introduce an analogous distance for multiplier algebras, that quantifies how far two algebras are from being completely isometrically isomorphic. 
We show that, in the setting of finite dimensional quotients of the Drury-Arveson space, two spaces are ``close" to one another if and only if their multiplier algebras are ``close", and that this happens if and only if one of the underlying point sets is close to an image of the other under a biholomorphic automorphism of the unit ball. 
These equivalences are obtained as corollaries of quantitative estimates that we prove. 
\end{abstract}

\maketitle

\section{Introduction}

The objective of this paper is to study the relationship between the structure of a reproducing kernel Hilbert space (RKHS), the structure of its multiplier algebra, and the geometry of the underlying set. 
We shall use freely basic facts on RKHSs which can be found in standard introductions; see \cite{AMBook, Paulsen-Raghupati-Book2016}. 
This continues a standing research program on the  classification of complete Pick spaces and algebras; see \cite{APV2003,ARS2008,Davidson-Hartz-Shalit2015,Davidson-Ramsey-Shalit-2011,Davidson-Ramsey-Shalit-2015,Hartz2012,Hartz2017,HL2018,KerMcSh2013,McSh2017,OS2020,Salomon-Shalit2016}. 
More generally, our study fits into a general theme in the geometric study of RKHSs; see, e.g., \cite{ARSW2011,CD78,Mc96,Rochberg86} for a sample of results, or \cite{Rochberg2017,Rochberg2019,Rochberg2020,Salomon-Shalit-Shamovich2018,Salomon-Shalit-Shamovich2020} for several recent examples in the commutative and the noncommutative setting. 

\subsection{Background on the isomorphism problem}\label{subsec:background}
For $d \in \bN \cup \{\infty\}$, the Drury-Arveson space $H^2_d$ is the RKHS on the open unit ball $\bB_d$ of a complex $d$-dimensional Hilbert space, determined by the kernel $k(z,w) = \frac{1}{1 - \langle z, w \rangle}$. 
We write $\cM_d$ for the multiplier algebra $\Mult(H^2_d)$ of $H^2_d$. 
The space $H^2_d$ can be considered as a module over $\cM_d$ and plays a prominent role in the theory of {\em Hilbert modules}  (see the survey \cite{ShalitSurvey2014}). 
A subspace of the form $\cH_X = \overline\spn\{k_\lambda: x \in X\} = H^2_d\big|_X$, where $X \subseteq \bB_d$, is called a {\em quotient} of $H^2_d$ (the reason for this terminology is that $\cH_X$ can be naturally identified with the quotient of $H^2_d$ by the submodule of functions vanishing on $X$).

By Theorem 4.2 in \cite{AM2000}, every complete Pick RKHS is isometrically isomorphic as an RKHS to (i.e., is a rescaling of) a quotient space of the Drury-Arveson space. 
Thus, the study of complete Pick spaces and their multiplier algebras can to a large extent be carried out for quotients of $H^2_d$. 
It follows from the complete Pick property that $\cM_X  := \Mult(\cH_X) = \cM_d\big|_X$. 
By Proposition 2.2 in \cite{Davidson-Ramsey-Shalit-2015}, it suffices to consider only $X$ which are {\em multiplier varieties}, which means that $X$ is the joint zero set of an ideal of multipliers in $\cM_d$. 

With the comments in the last paragraph in mind, it is natural to ask in what way exactly does a multiplier variety $V \subseteq \bB_d$ determine the structure of $\cH_V$ and of $\cM_V$. 
By Theorems 4.4 and 5.10 in \cite{Davidson-Ramsey-Shalit-2015}, $\cM_V$ is isometrically isomorphic to $\cM_W$ if and only if $V$ and $W$ are {\em congruent}, meaning that there is a biholomorphic automorphism $\phi$ of $\bB_d$ such that $\phi(V) = W$, and this happens if and only if the multiplier algebras are completely isometrically isomorphic. 
The methods employed (e.g., \cite[Proposition 4.1]{Davidson-Ramsey-Shalit-2015}) also show that $\cH_V$ is isometrically isomorphic as a RKHS to $\cH_W$ if and only $V$ and $W$ are congruent. 

Next, one can ask: if $\cM_V$ and $\cM_W$ are merely isomorphic, how are then $V$ and $W$ related? 
By \cite[Theorem 5.6 ]{Davidson-Ramsey-Shalit-2015} under some reasonable regularity assumptions on $V$ and $W$, an isomorphism from $\cM_V$ onto $\cM_W$ implies the existence of a biholomorphism from $W$ onto $V$. 
Moreover, by \cite[Theorem 6.2]{Davidson-Hartz-Shalit2015} this biholomorphism must be bi-Lipschitz with respect to the pseudohyperbolic metric. 
Under some additional assumptions, a biholomorphism from $W$ onto $V$ gives rise to an isomorphism from $\cM_V$ to $\cM_W$; this has been proved, e.g., when the varieties are homogeneous \cite{Davidson-Ramsey-Shalit-2011} or finite Riemann surfaces \cite{KerMcSh2013}. 

\subsection{What this paper is about}

This paper is strongly inspired by two recent papers of Rochberg, \cite{Rochberg2017} and \cite{Rochberg2019}. 
In \cite{Rochberg2017}, a certain metric induced by RKHSs is used to show that the Dirichlet space is not a quotient of the Drury-Arveson space $H^2_d$ for finite $d$ (it is worth noting that the same methods can be modified to show that the Dirichlet space is not even boundedly isomorphic to a quotient of $H^2_d$ for finite $d$ \cite{Hartz16p}). 
In \cite{Rochberg2019}, the special case of complete Pick spaces on finitely many points is carefully investigated;  the natural metric and additional geometric constructions are studied, and shown to hold much information on the related function spaces. 
In particular it is shown in \cite[Theorem 7]{Rochberg2019} that if $X_1$ and $X_2$ are finite sets of points in the unit ball, then $\cH_{X_1}$  and $\cH_{X_2}$ are isometrically isomorphic as RKHSs if and only if $X_1$ and $X_2$ are congruent.

In this paper we follow Rochberg in focusing on finite sets of points $X_1, X_2$ considered as metric spaces, but here we are most interested in the question {\em what happens when $X_1$ and $X_2$ are not exactly congruent, but are ``close" to being so}. 
Are then $\cH_{X_1}$ and $\cH_{X_2}$ forced to be, in some sense, ``close" to being isometrically isomorphic as reproducing kernel Hilbert spaces? 
And conversely, if the function spaces are ``close" to being isometrically isomorphic in some sense, are the sets then ``close" to be being biholomorphic images one of the other, or ``close" to being isometric? 
Similarly, we are interested in analogous questions of how the multiplier algebras, rather than the RKHSs, are determined by the underlying sets and their geometry, and how small changes in the structure of the underlying sets are reflected in small changes in the structure of the algebra. 

In other words, this paper studies in a quantitative manner the relationship between the the structure of function spaces and operator algebras on sets, on the one hand, and the geometry of the underlying sets, on the other hand. 

\subsection{Main results}
To present the main results of this paper we need to briefly introduce some more terminology. 
Suppose that $\rho_{\rm{ph}}$ is the pseudohyperbolic metric on the unit ball $\bB_d$. 
This metric induces the Hausdorff metric $\rho_H$ on subsets of the ball (see Definition \ref{def:Hausdorff-Symmetric-Distances}), which in turn gives rise to an automorphism invariant Hausdroff distance between subsets
\[
\tilde{\rho}_{H}(X,Y) =\inf\left\{\rho_{H}(X, \Phi(Y)):\Phi\in\Aut(\mathbb{B}_d)\right\};
\]
see Definition \ref{def:automorphism-invariant-Hausdorff-Symmetric-Distances}.
We introduce a distance function $\rho_{RK}$ which is analogous to the Banach-Mazur distance, defined as follows
\[
\rho_{RK}(\cH_1,\cH_2) = \log \left(\inf\left\{\|T\| \|T^{-1}\| : T : \cH_1 \to \cH_2 \textrm{ is an RKHS isomorphism}\right\} \right) .
\]
This distance function quantifies how far two spaces are from being isometrically isomorphic as reproducing kernel Hilbert spaces. 
Similarly, we introduce a distance function $\rho_{M}$, defined via a similar formula, that quantifies how far two multiplier algebras are from being completely isometrically isomorphic as multiplier algebras (see Definitions \ref{def:RKBMdist} and \ref{def:BMmultdist}). 
Most of this paper is concerned with studying the reltionships between these distance functions.

Our principal theorem, Theorem \ref{thm:main}, says roughly that for two subsets $X = \{x_1, \ldots, x_n\}$ and  $Y = \{y_1, \ldots, y_n\}$ of the unit ball $\bB_d$ we have that $\tilde{\rho}_H(X,Y)$ is small if and only if $\rho_{RK}(\cH_X,\cH_Y)$ is small, and this happens if and only if $\rho_{M}(\cM_X,\cM_Y)$ is small.  

Theorem \ref{thm:main} is a direct consequence of explcit inequalities between the different metrics, which are the heart of this paper.
Sections \ref{sec:distance_spaces} and \ref{sec:spaces} contain the quantitative versions of each one of the implications in the theorem; see Propositions \ref{prop:BMspace_mult}, \ref{prop:BMspace_mult_converse}, \ref{prop:set_RKHS} and \ref{prop:RKHS_set}. 
These provide us with quantitative versions of results on the isomorphism problem that were summarized in Subsection \ref{subsec:background}. 
We hope that such quantitative results will eventually help to resolve open problems on isomorphisms of infinite dimensional quotients of the Drury-Arveson space. 

\begin{ack}
We would like to thank the referee for numerous helpful suggestions to enhance the organization of our ideas and improve the exposition of this article.
\end{ack}

\section{Preliminaries}\label{sec:pre}

We begin by recalling several notions from the theory of reproducing kernel Hilbert spaces. 
Let $\mathcal{H}$ be a reproducing kernel Hilbert space on a set $X$ with reproducing kernel $K$. 
We use $\Mult(\mathcal H)$ to denote the multiplier algebra of $\mathcal H$. 
By a \emph{kernel function} we mean a two-variable function $K:X\times X\to \mathbb{C}$ on a nonempty set $X$ that satisfies the following property: for every $n\in \mathbb{N}$ and for every choice of $n$ distinct points $\{x_1,...,x_n\}\subseteq X$, the matrix $[K(x_i,x_j)]$ is positive semidefinite. We will use the notation $K\geq 0$ to denote that the function $K$ is a kernel function. 
If $\mathcal{H}$ is an RKHS on $X$ with reproducing kernel $K$, then $K$ is, indeed, a kernel function. 
The element $k_y = K(\cdot,y) \in \cH$ is called \emph{the reproducing kernel at $y\in X$}.

\subsection{Notions of isomorphism for reproducing kernel Hilbert spaces}
In the study of reproducing kernel Hilbert spaces, it is useful to have notions for comparing two spaces and deciding whether or not and in what sense they are alike (see \cite[Section 2]{McSh2017} or \cite[Section 2.6]{AM2000}). 
Since the kinds of questions we ask are somewhat new, it is convenient for us to introduce some terminology, which is not completely aligned with what finds in the literature. 

For $i=1,2$, let $\mathcal H_i$ be reproducing kernel Hilbert spaces
respectively on sets $X_i$ with reproducing kernels $K_i(x,y)=k^i_y(x)$, and let $\mathcal M_i=\Mult(\mathcal H_i)$ be their multiplier algebras. 
The space $\mathcal{H}_2$ is said to be a \emph{rescaling} of $\mathcal H_1$ if there is a bijection $F:X_1\to X_2$ and a nowhere-vanishing complex valued function $x \mapsto \lambda_x$ defined on $X_1$ such that 
\[
K_1(x,y)= \overline{\lambda_x}\lambda_y K_2(F(x), F(y)) \,\,\text{ for all }\,\, x,y\in X_1.
\]
This happens if and only if there exists a unitary $U:\mathcal H_1\to \mathcal H_2$ that maps each one-dimensional space $\mathbb C k^1_x$ onto $\mathbb C k^2_{F(x)}$, that is, there exists a nowhere-vanishing complex valued function $x \mapsto \lambda$ defined on $X_1$ such that 
\[
Uk^1_x = \lambda_x k^2_{F(x)} \,\,\,\,\text{ for all }\,\,\,\, x\in X_1.
\]
We shall refer to the unitary $U$ as an \emph{isometric isomorphism of reproducing kernel Hilbert spaces}, and say that $\mathcal H_1$ and $\mathcal H_2$ are \emph{isometrically isomorphic as reproducing kernel Hilbert spaces}. 
The point of this definition is that spaces which are isometrically isomorphic as RKHSs can be considered to be the same in every respect: the unitary between them preserves the function theoretic structure, in the sense that point evaluations are sent to point evaluations. 
While there is always a unitary map from one Hilbert space onto another of the same dimension, it is much harder for two RKHSs to be isometrically isomorphic as RKHSs (simple examples of pairs non isometrically isomorphic RKHSs of the same dimension follow very concretely from Proposition \ref{prop:RKHS_set} below).

We are interested in understanding a coarser and more flexible structure on RKHSs. An {\em isomorphism of reproducing kernel Hilbert spaces} from $\cH_1$ to $\cH_2$ (or simply an \emph{RKHS isomorphism}) is a bijective bounded linear map $T : \cH_1 \to \cH_2$ defined by  
\be\label{eq:Tiso}
T (k^1_x) = \lambda_x k^2_{F(x)} \,\,, \,\,  x \in X_1, 
\ee
where $x \mapsto \lambda_x$ is a nowhere-vanishing complex valued function and $F: X_1 \to X_2$ is a bijection. 
If there is an RKHS isomorphism from $\cH_1$ to $\cH_2$, then clearly there is an inverse which too is an RKHS isomorphism, and we then say that $\cH_1$ and $\cH_2$ are {\em isomorphic as RKHSs}, or simply that they are {\em isomorphic}.  
The point is that the linear isomorphism between the spaces preserves point evaluations, and hence the function theoretic structure. 

It is perhaps worth to warn the reader that, as far as we know, this notion of isomorphism between RKHSs --- the one described above via Equation (\ref{eq:Tiso}) --- is not commonly studied, and we were led to introduce it in order to avoid ambiguity and maintain the consistency in our definitions so that we could keep the logical flow of the paper smooth. In this regard, our terminologies differ from other places. For instance, in \cite[Section 2]{McSh2017}, due to reasons explained therein, the phrase ``isomorphic RKHSs" is used differenty and it refers to what we call isometrically isomorphic RKHSs. 
We also stress that the notion of ``isometric isomorphism" as defined above is precisely the same as that of ``rescaling"; we prefer this terminology because we also need to consider a weaker notion of isomorphism, the ``RKHS isomorphism", and this choice helps us emphasize the difference.

\subsection{Notions of isomorphism for multiplier algebras}\label{subsec:isomult}
Next, we want to consider what an isomorphism of multiplier algebras should be. First, to be called a {\em morphism} between multplier algebras, a map $\varphi : \cM_1 \to \cM_2$ will be required to preserve the algebraic structure, that is, to be an algebra homomorphism. 
Additionally, since multiplier algebras are algebras of functions on sets, a natural requirement for a morphism in this category is that point evaluations in $\cM_2$ are pulled back to point evaluations in $\cM_1$. 
Concretely, this means that $\varphi$ is implemented by composition with a function $G : X_2 \to X_1$, i.e. $\varphi(f) = f \circ G$.
Note that this isn't actually a separate assumption as every such map must be a homomorphism.  
Conversely, there are certain conditions under which an isomorphism of multiplier algebras can only arise as a composition with a bijection; this is the case in the finite dimensional case considered in this paper, and also holds more generally (see \cite{Salomon-Shalit2016}).

Finally, the multiplier algbera $\cM_i$ can identified with a concrete operator algebra on $\cH_i$ by identifying every multiplier $f \in \cM_i = \Mult(\cH_i)$ with the multiplication operator $M_f : h \mapsto fh$, $h \in \cH_i$. 
Recall that given a map $\varphi : \cM_1 \to \cM_2$ between operator algebras, the {\em completely bounded norm} (or {\em cb norm}, for short) is defined to be $\|\varphi\|_{cb} = \sup_n \|\varphi^{(n)}\|$, where $\varphi^{(n)}$ is the map $\varphi^{(n)} : M_n(\cM_1) = \cM_1 \otimes M_n \to M_n(\cM_2) = \cM_2 \otimes M_n$ between the matrices over $\cM_1$ and $\cM_2$, given by $\varphi^{(n)} = \varphi \otimes {\bf id}_{M_n}$ (see \cite{Pau02}). 
A map $\varphi$ is said to be {\em completely bounded} if $\|\varphi\|_{cb}<\infty$ and {\em completely isometric} if $\varphi^{(n)}$ is isometric for all $n$. 
In the theory of (not-necessarily selfadjoint) operator algebras, experience has taught us that the natural morphisms are the {\em completely bounded homomorphisms}. 
We say that $\varphi$ is a {\em complete isomorphism} if it is a bijective completely bounded isomorphism with a completely bounded inverse. 
In the finite dimensional case every isomorphism is a complete isomorphism, but keeping in mind the appropriate notion of morphisms can serve to guide us to discover the correct theorems and their proofs. 

After the above discussion on what are natural conditions that an isomorphism of multiplier algebras should have, we define a {\em multiplier algebra isomorphism} between multiplier algebras $\cM_1 = \Mult(\cH_1)$ and $\cM_2 = \Mult(\cH_2)$ to be a complete isomorphism $\varphi: \cM_1 \to \cM_2$ that is implemented as
\[
\varphi(f) = f \circ G \, \, , \,\, f \in \cM_1, 
\]
where $G : X_2 \to X_1$ is a bijection. 
If such an isomorphism exists then we say that $\cM_1$ and $\cM_2$ are {\em isomorphic as multiplier algebras}. 
If $\varphi$ is completely isometric then we say that $\cM_1$ and $\cM_2$ are {\em completely isometrically isomorphic as multiplier algebras}. 

As a testimony to the naturality of the the above definition, we note that it is not hard to see that if $T : \cH_1 \to \cH_2$ is an isomorphism of RKHS  then the map $M_f \mapsto (T^*)^{-1} M_f T^*$ induces an isomorphism of multiplier algebras, $\varphi : \cM_1 \to \cM_2$ (see Proposition \ref{prop:BMspace_mult}).
The paper \cite{Davidson-Hartz-Shalit2015} contains a continuum of natural weighted Hardy spaces on the unit disc such that their multiplier algebras are non-isomorphic. 
It follows therefore that these spaces are not isomorphic as RKHSs.

Furthermore, if $T$ is an isometric isomorphism of RKHSs, then the induced isomorphism $\varphi : \cM_1 \to \cM_2$ is a completely isometric isomorphism of multiplier algebras.
On the other hand, a completely isometric isomorphism of multiplier algebras need not imply the existence of an isometric isomorphism between the spaces.
Indeed, the Hardy space and the Bergman spaces on the unit disc both have $H^\infty$ as their multiplier algebra, and in fact the identity map from $H^\infty$ to itself is a completely isometric isomorphism that is implemented by composition with the identity map on the disc. 
However, the Hardy and Bergman spaces on the disc are not isometrically isomorphic as RKHSs (to see this one can use the fact that the Hardy space has the complete Pick property while the Bergman space does not, see \cite{AMBook}).

\subsection{Some useful facts about isomorphisms of RKHSs and multiplier algebras on finite subsets of the ball}\label{subsec:usefulfacts}

We record some simple observations on the spaces $\cH_X = H^2_d\big|_X$ and their multiplier algebras $\cM_X$ as defined in the introduction, where $X \subseteq \bB_d$ is finite. 
As vector spaces both $\cH_X$ and $\cM_X$ can be identified with the space of complex functions on $X$. 

Two such finite dimensional spaces $\cH_X$ and $\cH_Y$ are isomomorphic as RKHSs if and only their multiplier algebras are isomorphic as multiplier algebras, and this happens precisely when $X$ and $Y$ have the same cardinality. 
Indeed, if $F: Y \to X$ is a bijection, then $k_y \mapsto k_{F(y)}$ extends to an isomorphism of RKHSs. 
If $T : \cH_X \to \cH_Y$ is an isomorphism of RKHSs, then we noted two paragraphs ago that conjugation with the adjoint of $T$ gives rise to an isomorphism of multiplier algebras $\varphi : \cM_X \to \cM_Y$ (note that in the finite dimensional setting all linear maps are bounded, and in fact completely bounded). 
Finally, every isomorphism (in the mere algebraic sense) $\varphi: \cM_1 \to \cM_2$ gives rise to a bijection $F : Y \to X$ that induces it. 
For completeness, we explain this below. 

We let $\fM(\cM_X)$ denote the maximal ideal space of $\cM_X$.
The maximal ideal space plays an important role in the isomorphism problem \cite{Salomon-Shalit2016}, but in the case of general subsets $X$ in the ball it can be quite a wild and unwieldy topological space. 
Luckily, in the case of finite sets all of the subtleties disappear. 
For every $x \in X$ there is a multiplicative linear evaluation functional $\pi_x \in \fM(\cM_X)$ given by 
\[
\pi_x(f) = f(x). 
\]
It is easy to see that $\fM(\cM_X) = \{\pi_x : x \in X\}$, and therefore we can identify $\fM(\cM_X)$ with $X$ by identifying the evaluation functional $\pi_x$ at a point $x$ with the point $x$. 

Suppose that $Y \subseteq \bB_d$ is a finite set, and that $F : Y \to X$ is a function. 
Then we can define $\varphi : \cM_X \to \cM_Y$ by 
\[
\varphi(f) = f \circ F
\]
for all $f \in \cM_X$. 
Clearly, $\varphi$ is a unital homomorphism, which is an isomorphism if and only if $F$ is bijective. 
Conversely, if $\varphi : \cM_X \to \cM_Y$ is a unital homomorphism, then we can define a map $\varphi^* : \fM(\cM_Y) \to \fM(\cM_X)$ by 
\[
\varphi^*(\pi) = \pi \circ \varphi
\]
for all $\pi \in \fM(\cM_Y)$. 
Since every multiplicative linear functional on $\cM_X$ can be identified with a unique point in the underlying set $X$, we obtain a map $F : Y \to X$. 
We leave the reader with the straightforward task of verifying that the maps $F \mapsto \varphi$ and $\varphi \mapsto F$ defined above are mutual inverses and in particular, if $F$ is the map determined by $\varphi^*$, then $\varphi(f) = f\circ F$.

\subsection{The pseudohyperbolic metric} 

\begin{definition}\label{def:phdistance}
The \emph{pseudohyperbolic metric} $\rho_{\rm{ph}}$ on the open unit ball $\bB_d$ is given by 
\[
\rho_{\rm{ph}}(z,w):=\left\|\Psi_w(z)\right\| = \left\|\Psi_z(w)\right\|, \hspace{0.5cm} z,w\in \mathbb{B}_d, 
\]
where $\Psi_w$ is the elementary automorphism 
\[
\Psi_w(z) = \frac{w - P_w z - (1-\|w\|^2)^{1/2} P_w^\perp z}{1 - \langle z, w \rangle} 
\]
that maps $w$ to $0$ and $0$ to $w$.
Here $P_w$ is the orthogonal projection onto the span of $w$ and $P_w^\perp = I-P_w$ (see \cite[Section 2.2]{RudinUnitBall}). In particular, when $d=1$, we have 
\[
\Psi_w(z)=\frac{w-z}{1-\overline wz} \,\, \text{ and } \,\, \rho_{\rm{ph}}(z,w) = \left | \frac{w-z}{1-\overline wz} \right |.
\]
\end{definition}
The pseudohyperbolic metric and the Euclidean metric are not equivalent, but they are equivalent on every ball $r \bB_d$ of radius strictly less than $1$. 
Therefore, when working in a compact neighborhood of a fixed finite subset of the ball it essentially makes no difference whether we work with the Euclidean or the pseudohyperbolic metric. 
The pseudohyperbolic metric is better suited to study the function spaces of interest to us, but sometimes it is convenient to use the Euclidean metric, and we shall make use of both.

\section{Distances between subsets of the ball}
In this section we will discuss several ways in which one can measure distances between subsets of the ball, and we shall record a few useful facts about them.

In a metric space $(M,\rho)$, we define, for a set $F \subseteq M$ and a point $x \in M$, the distance from $x$ to $F$ by $\rho(x,F) = \inf \{\rho(x,y) : y \in F\}$.

\begin{definition}\label{def:Hausdorff-Symmetric-Distances}
Let $X,Y$ be two sets in a metric space $(M,\rho)$. 
The {\em Hausdorff distance} (induced by $\rho$) between $X$ and $Y$ is defined to be
\[
\rho_{H}(X,Y) = \max\left\{\max_{x \in X} \rho(x,Y), \max_{y \in Y} \rho(y,X)  \right\}. 
\]
For $X = \{x_1, \ldots, x_n\}$ and $Y = \{y_1, \ldots, y_n\}$ we define their {\em symmetric distance} to be 
\[
\rho_s(X,Y) = \min_{\sigma \in S_n} \max\{\rho(x_i,y_{\sigma(i)}) : i=1, \ldots, n\} , 
\]
where $S_n$ is the symmetric group on $\{1, \ldots, n\}$. 
\end{definition}

\begin{lemma}\label{lemma:Hausdorff-leq-symmetric}
If $X = \{x_1, \ldots, x_n\}$ and $Y = \{y_1, \ldots, y_n\}$ are subsets of a metric space $(M,\rho)$, then $\rho_{H}(X,Y)\leq \rho_s(X,Y)$.
\end{lemma}
\begin{proof}
Given $x_i\in X$, observe that for each permutation $\sigma\in S_n$ we have $$\rho(x_i,Y)\leq \rho(x_i,y_{\sigma(i)}),$$ and hence, taking maximum over all $i$'s, we obtain  
$$\max_{i\in \{1,...,n\}}\rho(x_i,Y)\leq \max_{i\in \{1,...,n\}}\rho(x_i,y_{\sigma(i)}).$$ Since the above inequality holds for every permutation in $S_n$, we get 
\begin{equation}
\max_{i}\rho(x_i,Y)\leq \min_{\sigma\in S_n}\max_{i}\rho(x_i,y_{\sigma(i)})=\rho_s(X,Y).
\end{equation} 
Similarly, $\max_{j}\rho(y_j,X)\leq \rho_s(Y,X)=\rho_s(X,Y)$, thus
\[
\rho_{H}(X,Y)=\max\left\{ \max_{x \in X} \rho(x,Y), \max_{y \in F} \rho(y,X) \right\}\leq \rho_s(X,Y).
\]
\end{proof}

\begin{proposition}\label{prop:Hausdorff_symmetric}
Let $X = \{x_1, \ldots, x_n\}$ and $Y = \{y_1, \ldots, y_n\}$ be subsets of a metric space $(M,\rho)$ and set $\delta:=\min\{\rho(x_i,x_j) : i \neq j\}.$ If $\rho_{H}(X,Y) < \delta/2$, then $\rho_s(X,Y) = \rho_{H}(X,Y)$.
\end{proposition}
\begin{proof}
Since $\rho_{H}(X,Y) < \delta/2$, it follows that $\max_{x \in X} \rho(x,Y)<\delta/2$ which in turn implies that for every $x\in X$, $\rho(x,Y)<\delta/2$. 
This amounts to: for every $i\in \{1,...,n\}$, there exists a $j\in \{1,...,n\}$ such that $y_j\in B(x_i;\delta/2)$. But since $X$ and $Y$ are finite sets with same cardinality, we infer that $y_j$ is the only element of $Y$ that is contained in $B(x_i;\delta/2)$. Let us rename and denote this $y_j$ by $y_i$, so that for every $i\in \{1,...,n\}$, we have a unique $y_i\in B(x_i;\delta/2)$.
Then $\rho(x_i,Y) = \rho(x_i,y_i)$. 

Similarly, for every $y_j\in Y$, there is one and only one element of $X$ that is contained in $B(y_j;\delta/2)$, and since, from the arguments in the above paragraph, there is already $x_j$ in $X$ with $\rho(y_j,x_j)<\delta/2$, it  follows that $\rho(y_j,X)=\rho(y_j,x_j)$. 
Consequently, we get 
\[
\rho_{H}(X,Y)=\max\left\{\rho(x_i,y_i): i\in \{1,...,n\} \right\}.
\] 
But this is same as $\min_{\sigma \in S_n} \max\{\rho(x_i,y_{\sigma(i)}) : i=1, \ldots, n\} $; for if we choose any permutation other than the identity permutation, which leads to measuring the distance of $x_i$ from $y_i$, it would lean to measuring the distance of one of the the elements of $X$, say $x_i$, to a point in $Y$ that is not in $B(x_i;\delta/2)$ thereby forcing the value of $\max_i\{(x_i,y_{\sigma(i)})\}$ to be strictly greater than $\delta/2$.
This proves that $\rho_{H}(X,Y)=\rho_s(X,Y)$.
\end{proof}
\begin{corollary}
If $\rho_s(X,Y) < \delta/2$, then $\rho_s(X,Y) = \rho_{H}(X,Y)$.
\end{corollary}

The following example shows that in Proposition \ref{prop:Hausdorff_symmetric}, the constant $\delta/2$ cannot be replaced by a larger constant.

\begin{example}
Let $t,\delta>0$ and assume also that $\delta < t/10 < 1/10$. 
Consider two subsets $X=\{0,\delta, t+\delta/2\}$ and $Y=\{\delta/2, t, t+\delta\}$ in $\bR$ endowed with the usual metric. 
Since all the points are on a straight line, it is easy to see that $\rho_s(X,Y)=t-\delta$, while $\rho_{H}(X,Y)=\delta/2$. 

This example can be modified to show that no constant $c$ strictly larger than $\delta/2$ can work in Proposition \ref{prop:Hausdorff_symmetric} even when we restrict attention to the pseudohyperbolic metric on the disc. 
Indeed, from Definition \ref{def:phdistance} it follows $\rho_{\rm{ph}}(z,w)/|z-w| \to 1$ uniformly as $z,w \to 0$. 
Then for sufficiently small $t$ we will have $\rho_H(X,Y) \approx \delta/2 < c$ while $\rho_s(X,Y) \approx t-\delta$. 
\end{example}

Recall that $\Aut(\mathbb{B}_d)$ denotes the group of biholomorphic automorphisms (or simply, automorphisms) of the unit ball $\bB_d$, and that two subsets of the unit ball are said to be {\em congruent} if one is the image of the other under an automorphism. 
Automorphisms preserve the pseudohyperbolic metric \cite[Theorem 8.1.4]{RudinUnitBall}, hence congruent subsets of the ball are isometric as metric spaces. 
Moreover, the RKHSs induced on two finite subsets of the ball by the Drury-Arveson space are isometrically isomorphic if and only if they are congruent \cite[Theorem 7]{Rochberg2019}. 
For this reason, we would like a measure of distance between subsets of the ball that is blind to automorphisms. 

\begin{definition}\label{def:automorphism-invariant-Hausdorff-Symmetric-Distances}
Let $\rho$ be a metric on $\mathbb{B}_d$, and let $X$ and $Y$ be two subsets of $\mathbb{B}_d$. 
The {\em automorphism invariant Hausdorff distance} between $X$ and $Y$ is defined to be
\[
\tilde{\rho}_{H}(X,Y) =\inf\left\{\rho_{H}(X, \Phi(Y)):\Phi\in\Aut(\mathbb{B}_d)\right\}. 
\]
If $X = \{x_1, \ldots, x_n\}$ and $Y = \{y_1, \ldots, y_n\}$ are finite sets of the same cardinality, then we define their {\em automorphism invariant symmetric distance} to be 
\[
\tilde{\rho}_{s}(X,Y) = \inf\left\{\rho_s(X, \Phi(Y)):\Phi\in\Aut(\mathbb{B}_d)\right\}. 
\]
Because $\Aut(\mathbb{B}_d)$ is a group this expression is symmetric in $X$ and $Y$.
\end{definition}

\begin{lemma}\label{lemma:auto-Hausdorff-leq-auto-symmetric}
Let $\rho$ be a metric on $\mathbb{B}_d$. If $X$ and $Y$ are finite subsets of $\mathbb{B}_d$ with the same cardinality,  then $$\tilde{\rho}_{H}(X,Y)\leq \tilde{\rho}_{s}(X,Y).$$
\end{lemma}

\begin{proof}
Notice that for any $\Psi\in \Aut(\mathbb{B}_d)$, we have by Lemma \ref{lemma:Hausdorff-leq-symmetric}
\[
\tilde{\rho}_{H}(X,Y)\leq \rho_{H}(X,\Psi(Y))\leq \rho_s(X, \Psi(Y)). 
\]
Taking infimum over all automorphisms in $\mathbb{B}_d$ we get the result.
\end{proof}

\begin{proposition}\label{prop:auto-Hausdorff_auto-symmetric}
Let $X = \{x_1, \ldots, x_n\}$ and $Y = \{y_1, \ldots, y_n\}$ be subsets of the metric space $(\mathbb{B}_d,\rho)$, and set $\delta:=\min\{\rho(x_i,x_j) : i \neq j\}.$ If $\tilde{\rho}_{H}(X,Y) < \delta/2$, then $\tilde{\rho}_{s}(X,Y) = \tilde{\rho}_{H}(X,Y)$.
\end{proposition}
\begin{proof} 
By Lemma \ref{lemma:auto-Hausdorff-leq-auto-symmetric}, it suffices to show that $\tilde{\rho}_{s}(X,Y)\leq \tilde{\rho}_{H}(X,Y)$. 
Let us write $D = \tilde{\rho}_{H}(X,Y)$, and let $\varepsilon > 0$ be small enough so that $D+\varepsilon < \delta/2$. 
By definition of $D$, we can find an automorphism $\Psi_\varepsilon\in \Aut(\mathbb{B}_d)$ such that 
\[
\rho_{H}(X,\Psi_\varepsilon(Y))< D+\varepsilon < \delta/2, 
\]
which, by Proposition \ref{prop:Hausdorff_symmetric}, yields $\rho_{s}(X,\Psi_\varepsilon(Y))=\rho_{H}(X,\Psi_\varepsilon(Y))$. 
Thus, 
\[
\tilde{\rho}_{s}(X,Y) = \inf\left\{\rho_s(X, \Phi(Y)):\Phi\in\Aut(\mathbb{B}_d)\right\} < D+\varepsilon. 
\]
Since this is true for all $\varepsilon$ as above we conclude that $\tilde{\rho}_{s}(X,Y)\leq \tilde{\rho}_{H}(X,Y)$.
\end{proof}

\begin{corollary}
Let $\rho$ be a metric on $\mathbb{B}_d.$ 
Let $X$ and $Y$ be finite subsets of $\mathbb{B}_d$ with the same cardinality and let $\delta=\min\{\rho(x_i,x_j) : i \neq j\}$. If $\tilde{\rho}_{s}(X,Y) < \delta/2$, then $\tilde{\rho}_{s}(X,Y) = \tilde{\rho}_{H}(X,Y)$.
\end{corollary}
Thus, once sets are close enough, $\tilde{\rho}_{s}$ and $\tilde{\rho}_{H}$ are the same. 
We choose to continue working $\tilde{\rho}_{s}$ because it is more convenient, keeping in mind the equivalence.

We close this section with an additional suggestion for measuring distances between subsets of the ball which might be relevant to our endeavours. 
A function $F = (F_1, \ldots, F_m) \in \bB_d \to \bC^m$ will be called a {\em vector valued multiplier} if its component functions are all multipliers, i.e., if $F_1, \ldots, F_m \in \cM_d$. 
We then define the {\em multipier norm} of a vector valued multiplier to be $\|F\| = \|M_F\|$, where we think of $M_F$ as the ``row multiplication operator" $M_F = (M_{F_1}, \ldots, M_{F_m}) : H^2_d \oplus \cdots \oplus H^2_d \to H^2_d$ that acts as
\[
M_F (h_1 \oplus \cdots \oplus h_m) = \sum_{i=1}^m F_i h_i . 
\]
If $X, Y$ are both finite subsets of the unit ball $\bB_d \subseteq \bC^d$ with $n$ points each, one can always find vector valued multipliers  (in fact, vector valued polynomials) $F,G : \bB_d \to \bC^d$ such that $F\big|_X$ is a bijection onto $Y$ and $G\big|_Y$ is its inverse.
This allows us to define the following measure of difference. 

\begin{definition}\label{def:dm}
Let $X = \{x_1, \ldots, x_n\}$ and $Y = \{y_1, \ldots, y_n\}$ be two subsets of the unit ball $\bB_d$. 
The {\em multiplier discrepancy} between $X$ and $Y$ is defined to be
\[
\delta_{mult}(X,Y) = \inf\left\{\max\{\|G\|, \|F\|\} : F,G \textrm{ are multipliers }, G\big|_Y =  \left(F\big|_X \right)^{-1}\right\}. 
\]
\end{definition}

One can show that $\delta_{mult}(X,Y) \geq 1$ whenever $n>1$ (if $X$ and $Y$ are singleton sets they can be mapped from one to the other using a constant function which may even be zero). 
If $X$ and $Y$ are congruent then $\delta_{mult}(X,Y) = 1$, and a weak compactness argument, together with the methods of \cite[Section 4]{Davidson-Ramsey-Shalit-2015} shows the converse. 

\begin{remark}
One might think that the following could be a good way to measure distance between sets in our context:
\[
\tilde{\delta}_{mult}(X,Y) \stackrel{?}{=} \inf\left\{\|G\| \|F\| : F,G \textrm{ are multipliers }, G\big|_Y =  \left(F\big|_X \right)^{-1}\right\}. 
\]
However, the above definition is not promising. For example, consider $X = \{0,r\}$. Then $\tilde{\delta}_{mult}(X,cX) = 1$ for every $c>0$, while $\tilde{\rho}_H(X,{cX})$ grows with diminishing $c$.
\end{remark}

\section{Distances for reproducing kernel Hilbert spaces and multiplier algebras}\label{sec:distance_spaces}

In this section we define notions of distance on the sets of all RKHSs and all multiplier algebras, and we describe a quantitative relation between the two notions in the case of complete Pick spaces. 

\begin{definition}\label{def:RKBMdist}
For $i=1,2$, let $\cH_i$ be a reproducing kernel Hilbert space on a set $X_i$ with kernel $K_i$. 
Define
\[
\delta_{RK}(\cH_1,\cH_2) = \inf\left\{\|T\| \|T^{-1}\| : T : \cH_1 \to \cH_2 \textrm{ is an RKHS isomorphism}\right\}.
\]
The {\em reproducing kernel Banach-Mazur distance} between these spaces is defined to be 
\[
\rho_{RK}(\cH_1,\cH_2) = \log \left( \delta_{RK}(\cH_1,\cH_2) \right). 
\]
\end{definition}
The reproducing kernel Banach-Mazur distance is always defined when $\cH_1,\cH_2$ are finite dimensional and of the same dimension, and in this case one can show that $\rho_{RK}(\cH_1,\cH_2) = 0$ if and only if $\cH_1$ and $\cH_2$ are isometrically isomorphic as RKHSs. 
If there is no RKHS isomorphism $T : \cH_1 \to \cH_2$ then one can say that $\rho_{RH}(\cH_1,\cH_2) = \infty$, but we will not deal with this case in this paper. 

If $\varphi : \cM_1 \to \cM_2$ is a homomorphism between multiplier algebras, then we let $\|\varphi\|_{cb}$ denote the completely bounded norm of $\varphi$. 
This makes sense, because multiplier algebras are operator algebras (see Section \ref{subsec:isomult}). 
\begin{definition}\label{def:BMmultdist}
Given two RKHSs $\cH_i$ ($i=1,2$) 
with multiplier algebras $\cM_i = \Mult(\cH_i)$ we define 
\[
\delta_{M}(\cM_1,\cM_2) = \inf\left\{\|\varphi\|_{cb} \|\varphi^{-1}\|_{cb} :\varphi : \cM_1 \to \cM_2 \textrm{ a multiplier algebra isomorphism}\right\}. 
\]
The {\em multiplier Banach-Mazur distance} between $\cM_1$ and $\cM_2$ is defined to be
\[
\rho_{M}(\cM_1,\cM_2) = \log \left(\delta_{M}(\cM_1,\cM_2) \right). 
\]
\end{definition}
Again, this distance is always defined when $\cH_1$ and $\cH_2$ are of the same finite dimension. 
In the case of finite dimensional quotients of $H^2_d$ that we are interested in, every isomorphism is a multiplier algebra isomorphism (see Section \ref{subsec:usefulfacts}), so we have in this case
\[
\delta_{M}(\cM_1,\cM_2) = \inf\left\{\|\varphi\|_{cb} \|\varphi^{-1}\|_{cb} :\varphi : \cM_1 \to \cM_2 \textrm{ is an algebra isomorphism}\right\}. 
\]

\begin{remark}
Some readers might wonder why we work with the completely bounded norm of the map $\varphi$ and $\varphi^{-1}$ rather than the operator norm of these maps, and whether we can obtain similar results if we work with the operator norm instead. 
The completely bounded norm is a natural choice for operator algebras, as we briefly discussed in Section \ref{subsec:isomult}; for a reader not familiar with this idea, we recommend \cite{Pau02}. 
We do not know whether our results hold with the norm instead of the completely bounded norm. 
\end{remark}

The goal of this section is to study the connection between $\rho_{RK}$ and $\rho_M$. 
We start with the easy direction, which holds true for arbitrary RKHSs. 
\begin{proposition}\label{prop:BMspace_mult}
For every pair of RKHSs $\cH_1$ and $\cH_2$, it holds that 
\[
\rho_{M}(\cM_1,\cM_2) \leq \rho_{RK}(\cH_1,\cH_2)^2 .
\]
\end{proposition}
%
\begin{proof}
Let $T:\cH_1\to \cH_2$ be an RKHS isomorphism given by \eqref{eq:Tiso}, that is, $T (k^1_x) = \lambda_x k^2_{F(x)}$ for every $x \in X_1$. 
Given $f\in \cM_1$, we have the multiplication operator $M_f$ on $\cH_1$. 
Consider the bounded operator $(T^*)^{-1}M_fT^*:\cH_2\to \cH_2$. 
We claim that it is a multiplication operator on $\cH_2$, i.e., there exists a function $h\in \cM_2$ such that $(T^*)^{-1}M_fT^*=M_h$. 
To this end, observe that for every $x\in X_1$, we have 
\[
TM_f^*(k_x^1)
=\lambda_x\overline{f(x)}k^2_{F(x)}
=\lambda_xM^*_{f\circ F^{-1}}(k^2_{F(x)})=M^*_{f\circ F^{-1}}(\lambda_xk^2_{F(x)})
=M^*_{f\circ F^{-1}}T(k^1_x),
\]  
which yields $TM_f^*T^{-1}=M^*_{f\circ F^{-1}}$, and our claim is proved.
Thus $h=f\circ F^{-1}$ is a multiplier, which leads us to define $\varphi:\cM_1 \to \cM_2$ via $$\varphi(f)=f\circ F^{-1}.$$ It is then easy to see that $\|\varphi\|_{cb} \leq \|T\| \|T^{-1}\|$. The conclusion $\rho_M(\cM_1,\cM_2) \leq \rho_{RK}(\cH_1,\cH_2)^2$ follows readily, since $\varphi^{-1}$ is also given by composition and is implemented by conjugation with $T$. 
\end{proof}

The following two lemmas will be used to prove a converse to Proposition \ref{prop:BMspace_mult}. 
\begin{lemma}\label{lem:Schursum1}
Let $v_1, \ldots, v_n$ be distinct elements in the unit ball $\bB_d \subseteq \bC^d$. 
For all $\varepsilon > 0$ and $r\in (0,1)$, there exists $N \in \bN$ such that the following matrix inequality holds: 
\[
\varepsilon\left[\sum_{k=0}^N \langle v_i, v_j \rangle^k \right] \geq  \left[\sum_{k=N+1}^\infty \langle u_i, u_j \rangle^k\right]
\]
for every $n$-tuple $u_1, \ldots, u_n \in r\bB_d$. 
\end{lemma}

\begin{proof}
Fix $\varepsilon>0$ and let $A=\left[\langle v_i, v_j \rangle \right]_{i,j=1}^n$ be the  $n\times n$ Grammian matrix of vectors $\{v_1,...,v_n\}$. 
Similarly let $u_1, \ldots, u_n \in r\bB_d$ and put $B=\left[\langle u_i, u_j \rangle \right]_{i,j=1}^n$. 

If $T$ is a linear operator on $\bC^n$, then we let $\operatorname{Im}(T) = \{Tv : v \in \bC^n\}$ denote the image of $T$.
For an arbitrarily chosen natural number $M$ we define
\[
\cG_M:=\operatorname{Im}\left(\sum_{k=0}^M A^{\circ k} \right), 
\]
where $A^{\circ k}$ is the Schur product of $A$ with itself $k$ times.
These form an increasing sequence $\cG_1\subseteq \cG_2 \subseteq ... \subseteq \mathbb{C}^n$ of subspaces of $\mathbb{C}^n$, and so they stabilize after a certain point. Thus, there exists a natural number $M_0$ such that 
\[
\cG_1\subseteq \cG_2 \subseteq ... \subseteq \cG_{M_0} = \cG_{M_0+1}= \cG_{M_0+2}=... \subseteq \mathbb{C}^n,
\]
or, in other words,
\[
\cG_M = \cG_{M_0} \subseteq \mathbb{C}^n \,\,\text{ for all }\,\, M \geq M_0.
\]
Let $K$ be the Szego kernel in the ball with kernel function $k_w(z) = (1 - \langle z, w \rangle)^{-1}$. 
Since the kernel functions $k_{v_1}, \ldots, k_{v_n}$ are linearly independent, the matrix 
\[
\left[K(v_i, v_j)\right] = \left[\frac{1}{1- \langle v_i, v_j \rangle}\right] = \sum_{k=0}^\infty A^{\circ k} 
\]
is invertible, so we have that $\cG_{M} = \mathbb{C}^n$ for all $M \geq M_0$. 

From $\max\{\|u_i\|^2 : i=1, \ldots, n\} \leq r^2 < 1$ we have the elementary estimate $\|B^{\circ k}\| \leq nr^{2k}$ for the norm of $B^{\circ k}$,  which yields the operator inequality $B^{\circ k} \leq n r^{2k} \cdot I$ for all $k$. 
Consequently,  
\[
\sum_{k=M+1}^\infty B^{\circ k} \leq \frac{nr^{2(M+1)}}{1-r^2} I \,\,\text{ for all }\,\, M \in \bN.
\]
On the other hand, since $\sum_{k=0}^{M_0} A^{\circ k}$ is strictly positive definite, there exists $c>0$ such that 
\[
\sum_{k=0}^{M} A^{\circ k} \geq c I \,\,\text{ for all }\,\,M \geq M_0.
\]
Now, we choose a natural number $N \geq M_0$ which satisfies 
\[
\frac{nr^{2(N+1)}}{1-r^2} \leq c \varepsilon.
\] 
This yields
\[
\varepsilon\sum_{k=0}^{N} A^{\circ k} \geq\sum_{k=N+1}^\infty B^{\circ k},
\]
 which is the desired inequality.  
\end{proof}

\begin{lemma}\label{lem:Schursum}
Let $v_1, \ldots, v_n$ be distinct elements in the unit ball $\bB_d \subseteq \bC^d$. 
For all $\varepsilon > 0$, there exists $N \in \bN$ such that the following matrix inequality holds: 
\[
\left[\sum_{k=0}^N \langle v_i, v_j \rangle^k \right] \geq \frac{1}{1+\varepsilon}  \left[\sum_{k=0}^\infty \langle v_i, v_j \rangle^k\right] = \frac{1}{1+\varepsilon}  \left[\frac{1}{1- \langle v_i, v_j \rangle}\right] .
\]
\end{lemma}

\begin{proof}
Fix $\varepsilon>0$ and let $A=\left[\langle v_i, v_j \rangle \right]_{i,j=1}^n$ be the  $n\times n$ Grammian matrix as above. 
Then we need to find an $N\in \mathbb{N}$ such that the following inequality holds:
\[
\sum_{k=0}^N A^{\circ k} \geq \frac{1}{1+\varepsilon} \sum_{k=0}^\infty A^{\circ k}.
\]
Subtracting $\frac{1}{1+\varepsilon} \sum_{k=0}^N A^{\circ k}$ from both sides, we seek
an $N$ such that the following equivalent inequality holds:
\[
\varepsilon \sum_{k=0}^N A^{\circ k} \geq \sum_{k=N+1}^\infty A^{\circ k}. 
\]
Invoking Lemma \ref{lem:Schursum1} for $u_i = v_i$ ($i=1,\ldots,n$), the proof is complete. 
\end{proof}

%

Now we are ready to prove that closeness of the multiplier algebras implies closeness of the function spaces. 

\begin{proposition}\label{prop:BMspace_mult_converse}
Let $X_1$ be a finite subset of $\bB_d$, and let $\cH_1 = \cH_{X_1}$ and $\cM_1 = \mlt\cH_1$.  
Put 
\[
r := \max\{\rho_{\rm ph}(x,y) : x,y \in X_1\}
\]
and fix some $R \in (0,1)$. 
For every $\varepsilon > 0$ there exists $N_0 \in \bN$, which depends only on $\varepsilon$, on $R$ and on $X_1$, such that the following holds: for every $X_2 \subseteq \bB_d$, if $\delta_{M}(\cM_1,\cM_2) < r^{-1}R$, then 
\[
\delta_{RK}(\cH_1,\cH_2) \leq (1+\varepsilon) \delta_{M}(\cM_1,\cM_2)^{N_0},
\]
where $\cH_2 = \cH_{X_2}$ and $\cM_2 = \mlt\cH_2$.
\end{proposition}

\begin{proof}
Let us assume that $G : X_2 \to X_1$ is a bijection that induces via composition an isomorphism $\varphi : \cM_1 \to \cM_2$ such that 
\[
\|\varphi\|_{cb} \|\varphi^{-1}\|_{cb} < r^{-1}R. 
\]
Since the norm of an isomorphism is at least $1$, we can consider constants $C_1, C_2$ such that 
\be\label{eq:C1C2}
\|\varphi\|_{cb} \leq C_1, \|\varphi^{-1}\|_{cb}\leq C_2,\,\, \textrm{ and } \,\, C_1, C_2 <   r^{-1}R. 
\ee
We will find natural numbers $M,N$ (which do not depend on $G$) such that the map $T : \cH_2 \to \cH_1$ given by $T: k^2_x \mapsto k^1_{G(x)}$ satisfies $\|T\|\|T^{-1}\| \leq (1+\varepsilon)C_1^{N}C_2^M$. 
This will give the asserted estimate with $N_0 = \max\{M,N\}$. 

We can assume, without any loss of generality, that $0\in X_1 \cap X_2$ and that
$G(0)=0$ --- if this is not the case, then we apply an automorphism of the ball to the points, which induces a reproducing kernel Hilbert space isomorphism between the function spaces (see \cite[Section 4]{Davidson-Ramsey-Shalit-2015} for details).
Note that, under this assumption, we have 
\[
\max\{\|x\| : x \in X_1\} = r, 
\] 
since $\rho_{\rm ph}(x,0) = \|x\|$ and the pseudohyperbolic metric is invariant under automorphisms of the unit ball. 

For the first part of the estimate, we find $N$ such that the norm of the map $T: k^2_x \mapsto k^1_{G(x)}$ can be bounded as follows: 
\[
\|T\| \leq (1+\varepsilon)^{1/2} C_1^N .
\]
By assumption, for every matrix valued multiplier $f$ in the unit ball of $\cM_1 \otimes M_{m,n}(\bC)$, the composition $f \circ G$ is a multiplier in $\cM_2$ of norm at most $C_1$. 
This holds, in particular, for $f$ the identity map, which is a row contraction. 
Thus, $\hat G = f \circ G$, with $f$ being the identity map, is a vector valued multiplier of norm at most $C_1$. 
Let us write $G$ for $\hat G$ and $C$ for $C_1$. 
Consider the following kernel function on $X_2$
\[
K_{G/C}(x,y) := K_1(G(x)/C,G(y)/C) = \frac{1}{1 - C^{-2}\langle G(x), G(y) \rangle}.
\] 
Our immediate goal is to show that $K_2 \geq K_{G/C}$ (our proof of this fact is inspired by ideas from \cite[Section 3]{Hartz2017}). 
Since $G/C$ is a multiplier in $\mathcal M_2$ of norm at most one, it follows that 
\[
K_2/K_{G/C}=K_2(1-(G/C)^*(G/C))\geq 0,
\] 
that is, $K_2/K_{G/C}$ is a kernel function. 
By the assumption made above $G/C(0)=0$ and consequently, $K_{G/C}$ is normalized at $0$ (meaning that $K_{G/C}(x,0) = 1$ for all $x$), which in turn implies that $K_2/K_{G/C}$ is normalized at $0$. 
The Schur complement of the row and column corresponding to $0$ in $K_2/K_{G/C}$ is $K_2/K_{G/C}-1$, and so it follows that the function $K_2/K_{G/C}-1$ is a kernel function. 
Since 
\[
K_2-K_{G/C}= K_{G/C}\left(K_2/K_{G/C}-1\right),
\]
we infer, by the Schur product theorem, that 
\[
K_2 \geq K_{G/C}, 
\]
as required. 
Thus we have the inequality of $n \times n$ matrices
\be\label{eq:kernel_ineq}
\left[\sum_k \langle x_i, x_j \rangle^k \right] \geq \left[\sum_k C^{-2k} \langle G(x_i), G(x_j) \rangle^k \right], 
\ee
where $x_1, \ldots, x_n$ are the points in $X_2$, and $G(x_1), \ldots, G(x_n)$ are the points in $X_1$. 

To finish off the estimate for $\|T\|$ we use Lemma \ref{lem:Schursum}. 
Fixing $\varepsilon>0$, the lemma applied to the set $X_1 = \{v_i = G(x_i) : i=1, \ldots, n\}$ gives us an $N\in \bN$ such that 
\[
C^{-2N} \sum_{k=0}^N \langle G(x_i), G(x_j) \rangle^k \geq \frac{C^{-2N}}{1+\varepsilon} \sum_{k=0}^\infty \langle G(x_i), G(x_j) \rangle^k .
\]
On the other hand, 
\[
\sum_{k=0}^\infty C^{-2k} \langle G(x_i), G(x_j) \rangle^k \geq \sum_{k=0}^N C^{-2k} \langle G(x_i), G(x_j) \rangle^k \geq C^{-2N} \sum_{k=0}^N \langle G(x_i), G(x_j) \rangle^k ,
\]
so by combining with \eqref{eq:kernel_ineq} we conclude that the matrix inequality
\[
\left[K_1(G(x_i),G(x_j)) \right] \leq (1+\varepsilon) C^{2N} \left[K_2(x_i,x_j) \right]
\]
holds. 
This means that the map $T : k^2_x \mapsto k^1_{G(x)}$ extends to a linear map of norm at most $(1+\varepsilon)^{1/2}C_1^{N}$.

We next bound the norm of the map $T^{-1} : k^1_{G(x)} \mapsto k^2_x$.  
We will show that there is an integer $M$ such that $\|T^{-1}\| \leq (1+\varepsilon)^{1/2}C_2^M$.
Recall that we are assuming that composition with $G^{-1}$ gives rise to an isomorphism of cb norm less than or equal to $C_2$. 
Let us write $C$ for $C_2$, and using similar notation as before, we find that $K_1 \geq K_{{G^{-1}}/C}$, an equality which can be rewritten as
\[
\left[K_1(G(x_i), G(x_j))\right] = \left[\sum_k \langle G(x_i), G(x_j) \rangle^k \right] \geq \left[\sum_k C^{-2k} \langle x_i, x_j \rangle^k \right]. 
\]
As above, we find that $G^{-1}$ must be a multiplier of norm at most $C$. 
Since $G^{-1}$ maps $0$ to $0$, the Schwarz lemma implies that $\|x\|\leq rC$ for all $x \in X_2$. 
Invoking the assumption $C = C_2 < r^{-1}R$, we find that $\|x\|\leq R <1$ for all $x \in X_2$. 

Invoking Lemma \ref{lem:Schursum1} we find $M$ such that the matrix inequality $\sum_{k=M+1}^\infty \langle x_i, x_j \rangle^k \leq \varepsilon K_1(G(x_i), G(x_j))$ holds. 
Note that an $M$ can be found that works for all the sets of points satisfying $\max\{\|x_i\|^2 : i=1, \ldots, n\} \leq R^2$. 
But then we find 
\[
(1+\varepsilon) K_1(G(x_i), G(x_j)) \geq C^{-2M}\sum_{k=0}^M \langle x_i, x_j \rangle + \sum_{k=M+1}^\infty \langle x_i, x_j \rangle^k \geq C^{-2M} K_2(x_i,x_j). 
\]
We conclude that the matrix inequality
\[
\left[K_2(x_i,x_j) \right] \leq (1+\varepsilon) C^{2M} \left[K_1(G(x_i),G(x_j)) \right]
\]
holds, and this shows that $\|T^{-1}\| \leq (1+\varepsilon)^{1/2}C_2^M$. 
Combining the estimates for $\|T\|$ and $\|T^{-1}\|$, and recalling that the argument works for any pair of constants as in \eqref{eq:C1C2}, the proof is complete. 
\end{proof}

\begin{corollary}\label{cor:BMspace_mult_converse_gen}
If $\cH_1$ and $\cH_2$ are complete Pick spaces that live on sets $X_1,X_2$ with $n$ points each, then for every $\varepsilon > 0$ there exists $N \in \bN$, which depends only on $\varepsilon$ and on $\cH_1$, such that
\[
\delta_{RK}(\cH_1,\cH_2) \leq (1+\varepsilon) \delta_{M}(\cM_1,\cM_2)^{N}  ,
\]
whenever $\delta_{M}(\cM_1,\cM_2)$ is sufficiently close to $1$. 
\end{corollary}
\begin{proof}
If $\cH_i$ and $\cH'_i$ are isometrically isomorphic as RKHSs for $i=1,2$, then 
\[
\delta_{RK}(\cH_1,\cH_2) = \delta_{RK}(\cH'_1,\cH'_2)
\]
and also 
\[
\delta_{M}(\mlt(\cH_1),\mlt(\cH_2)) = \delta_{M}(\mlt(\cH'_1),\mlt(\cH'_2)),
\]
and so the result follows from \ref{prop:BMspace_mult_converse} and the universality of the Drury-Arveson space \cite{AM2000}. 
\end{proof}

\begin{remark}
We probably cannot expect to replace the bound in the above corollary with a simple formula that holds uniformly for all $\cH_1$, as in Proposition \ref{prop:BMspace_mult}.
It is worth noting the example at the end of Section 3 in \cite{Hartz2017}, where it is shown that two complete Pick spaces defined on the same (infinite) subvariety of the unit ball can have completely boundedly isomorphic (via the identity map) multiplier algebras, while being different as reproducing kernel Hilbert spaces. 
The point is that the maxim {\em complete Pick spaces are determined by their multiplier algebras} should be understood with care.
\end{remark}


\section{Structure of spaces and algebras versus geometry of point sets}\label{sec:spaces}

In this section we prove the main results of the paper, which tie the geometry of a set $X$ to the structure of the reproducing kernel Hilbert space $\cH_X$ and its multiplier algebra $\cM_X$ in a quantitative way. 

\begin{proposition}\label{prop:set_RKHS}
Let $X = \{x_1, \ldots, x_n\}$ and $Y = \{y_1, \ldots, y_n\}$ be two subsets of the unit ball $\bB_d$. 
Define $r = \max\{\|z\| : z \in X \cup Y\}$ and $A=[K_1(x_i,x_j)]$ and $B=[K_2(y_i,y_j)]$. 
Then
\[
\delta_{RK}(\cH_X,\cH_Y) \leq \left(1 + \frac{4nr(1-r^2)^{-2}}{\min\{\lambda_{\min}(A), \lambda_{\min}(B)\}} \rho_s(X,Y)\right)^2, 
\]
where $\lambda_{\min}(\cdot)$ denotes the minimal eigenvalue of a matrix, and $\rho$ denotes the Euclidean distance in the ball. 
\end{proposition}
\begin{proof}
Given that $X, Y \subseteq \bB_d$, let us write $\cH_1 = H^2_d\big|_{X}$ and $\cH_2 = H^2_d\big|_{Y}$ so that 
$$K_i(z,w) = \frac{1}{1 - \langle z, w \rangle}\,\,\text{ for }\,\,i=1,2.$$ 
Let us further assume that $\rho_s(X,Y)=\varepsilon.$
Then there exists a bijection $\sigma$ in $S_n$ such that 
$\max\{\rho(x_i,y_{\sigma(i)}):i=1,...,n\} = \varepsilon$ and we may as well assume that $\sigma$ is the identity.  
Consider the canonical bijection $F:X\to Y$ obtained from $\sigma$, the bijection given by $x_i\mapsto y_{\sigma(i)}=y_{i}$, and use it to define a linear map from $\cH_1$ to $\cH_2$ by
\[
T(k^1_{x_i})=k^2_{F(x_i)}=k^2_{y_i}. 
\]
Clearly, $T$ is bijective and bounded, thus $T$ is an RKHS isomorphism. 
We are interested in computing a bound on the norm of $T$. Let $\|A\|_2$ denote the Frobenius norm of a matrix $A$, that is $\|A\|_2 = \sqrt{\sum_{i,j}|a_{ij}|^2} = \operatorname{Tr}(A^*A)^{1/2}$.
We put $A=[K_1(x_i,x_j)]$ and $B=[K_2(y_i,y_j)]$, observe that $B-A\leq \lambda_{\max}(B-A)I$ and $\lambda_{\min}(A)I\leq A$. 
This implies that 
\[
B-A \leq \left(\frac{\lambda_{\max}(B-A)}{\lambda_{\min}(A)}\right) A.
\]
Consequently, 
\be\label{eq:BleqCA}
B= A+(B-A)\leq \left( 1+ \frac{\lambda_{\max}(B-A)}{\lambda_{\min}(A)}\right) A.
\ee
Let 
\be\label{eq:mu}
\mu=1+ \frac{\lambda_{\max}(B-A)}{\lambda_{\min}(A)}.
\ee
Then, using \eqref{eq:BleqCA} and \eqref{eq:mu}, for every $\alpha_1,...,\alpha_n \in \bC$, we have the following computation.
\begin{align*}
\left\|T\left(\sum_{j=1}^n\alpha_jk^1_{x_j}\right)\right\|^2
&=\left\|\sum_{j=1}^n\alpha_j k^2_{y_{j}}\right\|^2\\
&=\sum_{i,j=1}^n\overline\alpha_i\alpha_j K_2(y_{i},y_{j})\\
&\leq \mu\sum_{i,j=1}^n\overline\alpha_i\alpha_jK_1(x_i,x_j) \\
&= \mu\left\|\sum_{j=1}^n\alpha_jk^1_{x_j}\right\|^2.
\end{align*} 
It follows that $\|T\| \leq \mu$. 
To estimate $\mu$, we will use the inequality
\begin{align*}
\left|K_1(x_i,x_j) - K_2(y_i,y_j)\right| &=
\left| \frac{1}{1 - \langle x_i, x_j \rangle} - \frac{1}{1 - \langle y_i, y_j \rangle} \right| \\
&\leq \frac{1}{(1-r^2)^2} (2r \varepsilon + \varepsilon^2) \\
&\leq \frac{4r \varepsilon}{(1-r^2)^2} , 
\end{align*}
where the first inequality is an elementary estimate, and the second inequality follows from $\varepsilon = \max_{k=i,j} \|x_k - y_k\| \leq 2r$. 
Then 
\begin{align*}
\mu &\leq 1 + \frac{\lambda_{\max}(B-A)}{\lambda_{\min}(A)} \\
&\leq 1 + \frac{\|B-A\|_2}{\lambda_{\min}(A)} \\
&\leq 1 + \frac{\sqrt{\sum_{i,j=1}^n (K_2(y_{i},y_{j}) - K_1(x_i,x_j))^2}}{\lambda_{\min}(A)} \\
&\leq 1 + \frac{4nr(1-r^2)^{-2}}{\lambda_{\min}(A)} \varepsilon.
\end{align*}
By arguing in the reverse direction, we obtain the asserted estimate. 
\end{proof}

Proposition \ref{prop:set_RKHS} shows how closeness of the sets imply closeness of the corresponding reproducing kernel Hilbert spaces. 
Our next task is to prove the converse. 
For this, we require a technical lemma regarding the {\em Procrustes problem} which is of independent interest.  
We will denote the group of $d \times d$ unitary matrices by $\cU(d)$. 

\begin{lemma}\label{lemma:Procrustes}
Let $A,B\in M_{d\times n}(\mathbb{C})$ such that $||A^{*}A-B^{*}B||_{2}<\varepsilon.$
Then we have:
\[
\min_{W\in \cU(d)}\|A-WB\|_{2}^{2} \leq d(2\|A\|_{2}\varepsilon^{1/2}+\varepsilon).
\]
\end{lemma}

\begin{proof}
For $1\leq i\leq d$, let $\sigma_i(M)$ denote the $i$th singular value of a matrix $M$. 
By Paragraph 7.4.5 in \cite{HJ-2ndEd-2013} (or Example 7.4.8 of \cite{HJ-1stEd-85})
\[
\min_{W\in \cU(d)}\|A-WB\|_{2}^{2}=\sum_{i=1}^{d}\sigma_{i}^{2}(A)+\sigma_{i}^{2}(B)-2\sigma_{i}(AB^{*}).
\]
We will now show that the right hand side is less than or equal than $d(2\|A\|_{2} \varepsilon^{1/2}+\varepsilon)$, and that will complete the proof. 

Notice that $\sigma_{i}(A^{*}A) = \sigma_{i}(AA^{*})=\sigma_{i}^{2}(A)$ and similarly:
\[
|\sigma_{i}^{2}(AB^{*})-\sigma_{i}^{2}(AA^{*})|=|\sigma_{i}(AB^{*}BA^*)-\sigma_{i}(AA^*AA^*)|.
\]
Using the inequality $|x-y|\leq\sqrt{|x^{2}-y^{2}|}$, we get:
\[
|\sigma_{i}(AB^{*})-\sigma_{i}^{2}(A)| \leq {|\sigma_{i}(AB^{*}BA^*)-\sigma_{i}(AA^*AA^*)|}^{1/2}.
\]
Recall the fact that for any $X,Y\in M_{d\times n}(\mathbb{C})$ we have $|\sigma_i(X)-\sigma_i(Y)|\leq\|X-Y\|_2$ (see, e.g., \cite[Corollary 7.3.5]{HJ-2ndEd-2013}). 
We use this fact and the sub-multiplicativity of the Frobenius norm to estimate for all $1\leq i\leq d$:
\begin{align*}
\sigma_{i}^{2}(A)+\sigma_{i}^{2}(B)-2\sigma_{i}(AB^{*}) & \leq 2|\sigma_{i}^{2}(A)-\sigma_{i}(AB^{*})|+|\sigma_{i}^{2}(A)-\sigma_{i}^{2}(B)| \\
& \leq 2 {|\sigma_{i}(AB^{*}BA^*)-\sigma_{i}(AA^*AA^*)|}^{1/2} + \|A^*A -B^*B\|_2 \\
& \leq 2 {\|AB^{*}BA^*-AA^*AA^*\|_2}^{1/2} + \varepsilon\\
& \leq 2 \|A\|_2\|B^*B-A^*A \|_2^{1/2} +\varepsilon \\ 
&\leq 2 \|A\|_2\varepsilon^{1/2} +\varepsilon 
\end{align*}
as required. 
\end{proof}

\begin{proposition}\label{prop:RKHS_set}
Let $X = \{x_1,\dots,x_n\}$ be a finite subset of $\mathbb{B}_d$. Put
$r := \max_{i,j}\{\rho_{\rm ph}(x_i,x_j)\}$.
There exists a constant $C>0$ which depends only on $n,d$ and $r$, such that for any $Y=\{y_1,\dots,y_n \}$ with $\delta_{RK}(\mathcal{H}_X,\mathcal{H}_Y) < 2$:  
\[
\tilde{\rho}_{s}(X,Y) \leq C(\delta_{RK}(\mathcal{H}_X,\mathcal{H}_Y)-1)^{1/4},
\]
where $\tilde{\rho}_{s}$ denotes the automorphism invariant symmetric distance induced by the pseudohyperbolic metric $\rho_{\rm{ph}}$. 
\end{proposition}
\begin{proof}
Let $X,Y\subset\mathbb{B}_d$ be as in the statement of the proposition. Let  $2>\alpha >\delta_{RK}(\mathcal{H}_X,\mathcal{H}_Y)$. Then there exists  an RKHS isomorphism  $T: \mathcal{H}_X\rightarrow\mathcal{H}_Y$ that satisfies:
\be\label{eq:5.2.1}
\|T\|\|T^{-1}\|<\alpha.
\ee
We will prove this implies:
\[
\tilde{\rho}_{s}(X,Y) \leq\widetilde{C}(\alpha-1)^{1/4}.
\]
Where $\widetilde{C}>0$ is a constant which depends only on $r,n$ and $d$.
By taking the infimum of the right hand side of the inequality over all $2>\alpha >\delta_{RK}(\mathcal{H}_X,\mathcal{H}_Y)$ we obtain the result.

Assume $T$ is given by the formula:
\[
T(k_{x_i})= \lambda_i k_{y_i}.
\]
For some $\lambda\in\mathbb{C}^n$. 
Because all the quantities that appear in the conditions of the proposition (as well as the norm of $T$ and its inverse) are invariant under automorphisms of the ball, we may assume  $x_1 = y_1 = 0$ without loss of generality (see \cite[Section 4]{Davidson-Ramsey-Shalit-2015}). 
It now suffices to prove the inequality for the symmetric distance induced by the Euclidean metric, since the two metrics are equivalent inside the closed ball of radius $r<1$ centered at the origin. 

After rescaling $T$ by a complex number of unit modulus, we may assume that $\lambda_1 \geq 0$. 
After rescaling $T$ by a positive number, we may assume as well that $\|T\|=\|T^{-1}\|\leq\sqrt{{\alpha}}$. Under these assumptions, \eqref{eq:5.2.1} implies for any $x\in\mathcal{H}_X$ with $\|x\| = 1$:
\be\label{eq:5.2.2}
\alpha^{-1} \leq \|T^{-1}\|^{-2}\leq \|Tx\|^2 \leq \|T\|^2 \leq \alpha.
\ee
Because $ \alpha-1$ is greater then $1-\alpha^{-1}$, we have $|\langle(I-T^{*}T)x,x\rangle|=| \|x\|^{2}-\|Tx\|^{2}|\leq \alpha-1$. \\ Define $\varepsilon = \alpha-1$ to get
\be\label{eq: 5.2.5}
\|I-T^*T\|\leq\varepsilon .
\ee
Notice that for any $j$ we have $\rho(0,x_j)=\|x_j\|$, therefore:
\be\label{eq:5.2.4}
\max_i\{\|x_i\|\}\leq r.
\ee

We now bound the difference between the kernel functions. 
Denote $k(x_i,x_j)$ and $ k(y_i,y_j)$ by $k_{ij}^x$ and $ k_{ij}^y$ respectively. Furthermore, denote $\overline{\lambda_i} {\lambda_j}=\lambda_{ij} $.
For any $1\leq i,j \leq n$ we estimate:
\begin{align*}
|k_{ij}^{x}-\lambda_{ij}k_{ij}^{y}| & =|\langle(I-T^{*}T)k_{x_j},k_{x_i}\rangle| \\ 
&\leq\|I-T^{*}T\|\|k_{x_i}\|\|k_{x_j}\| \\
&\leq\frac{\varepsilon}{(1-\|x_{i}\|^{2})^{1/2}(1-\|x_{j}\|^{2})^{1/2}}
\end{align*}
and so, after putting $\delta = \frac{\varepsilon}{1-r^{2}}$,
\be\label{eq:star}
|k_{ij}^{x}-\lambda_{ij}k_{ij}^{y}| \leq \delta.
\ee
Dividing both sides of \eqref{eq:star} by $|k^y_{ij}k^x_{ij}|$ which is bounded from below by $\frac {1}{4}$ we get:
\begin{align*}
4\delta & \geq |1-\langle y_i,y_j\rangle-\lambda_{ij}(1-\langle x_i,x_j\rangle)| \\
& = |\langle x_{i},x_{j}\rangle-\langle y_{i},y_{j}\rangle-(1-\lambda_{ij})(\langle x_{i},x_{j}\rangle-1)|\\
& \geq |\langle x_i,x_j\rangle - \langle y_i,y_j\rangle| - |1-\lambda_{ij}|(|\langle x_i,x_j\rangle|+1).
\end{align*}
This implies that
\be\label{eq:5.2.3}
|\langle x_i,x_j\rangle - \langle y_i,y_j\rangle|\leq 4 \delta + 2|1-\lambda_{ij}|.
\ee
Our next goal is to bound $|1-\lambda_{ij}|$.  
Recalling the assumptions $x_1 = y_1 = 0$, $\lambda_1 \geq 0$, and plugging $i =1$ into \eqref{eq:star} we get 
\[
|1- \lambda_1 \lambda_j| \leq \delta, 
\]
and in particular, for $j=1$, 
\[
|1 - \lambda_1| \leq |1-\lambda_{1}^{2}| \leq \delta.
\]
Thus, using $|\lambda_j| \leq \|T\| \|k_{x_j}\|/\|k_{y_j}\| \leq 2\sqrt{\alpha}/(1-r^2)^{1/2}$, 
\begin{align*}
|\lambda_j - 1|
&\leq|\lambda_j - \lambda_{1} \lambda_j| + |\lambda_1 \lambda_j - 1| \\
&\leq |\lambda_j| \delta + \delta \leq 2\sqrt{\alpha}/(1-r^2)^{1/2}\delta + \delta \\
&\leq 3\sqrt{\alpha}/(1-r^2)^{1/2}\delta.
\end{align*}
Now we can finally estimate $|1-\lambda_{ij}| = |1 - \overline{\lambda_i} {\lambda_j}|$: 
\begin{align*}
|1-\lambda_{ij}|
&\leq |\overline{\lambda_i} {\lambda_j} -  \lambda_1 \lambda_j  | + |\lambda_1 \lambda_j - 1| \\
&\leq |\lambda_j| | {\lambda_i} - \lambda_1| + \delta \\
&\leq  12\alpha/(1-r^2)\delta + \delta \leq 13\alpha/(1-r^2) \delta. 
\end{align*}
By \eqref{eq:5.2.3}, we have for all $i$ and $j$:
\[
|\langle x_i,x_j\rangle - \langle y_i,y_j\rangle| \leq 4 \delta + 26\alpha/(1-r^2) \delta \leq \frac{30\varepsilon\alpha}{(1-r^2)^2}.
\]
Let $A$ and $B$ be the $d\times n$ matrices with columns given by the coordinate vectors of $x_1,\dots,x_n$ and $y_1, \dots,y_n$ respectively. 
This yields
\[
\|A^*A-B^*B\|_2 \leq C\varepsilon
\]
for $C=\frac{30\alpha n}{(1-r^2)^2}$. 
Because $\|A\|_2 = \sqrt{\sum\|x_i\|_2^2}< \sqrt{n}$, by Lemma \ref{lemma:Procrustes} there exists a unitary matrix $W\in \cU(d)$ for which:
\[
\|A-WB\|_2 \leq [d(2n^{1/2}(C\varepsilon)^{1/2}+C\varepsilon)]^{1/2} \leq 2 \sqrt{Cdn^{1/2}}\varepsilon^{1/4}.
\]
The last inequality holds because $\alpha<2$ and therefore $\varepsilon <1$, and $C \geq 1$. 
Letting $\widetilde{C}=2 \sqrt{Cdn^{1/2}}$ and recalling that $\varepsilon = \alpha - 1$, we get
\[
\tilde{\rho}_{s}(X,Y) = \tilde{\rho}_{s}(X,W(Y))\leq\widetilde{C}(\alpha-1)^{1/4}, 
\]
as we wanted to show. 
This finishes the proof.
\end{proof}

The following theorem summarizes the main results of this paper in qualitative form. 

\begin{theorem}\label{thm:main}
Fix $n \geq 2$. 
Let $X = \{x_1, \ldots, x_n\} \subset \bB_d$ and let $Y^{(k)} = \{y^{(k)}_1, \ldots, y^{(k)}_n\}$ be a sequence of subsets of $\bB_d$. 
Let $\tilde{\rho}_s$ be the automorphism invariant symmetric distance induced by the pseudohyperbolic metric $\rho_{\rm{ph}}$.
Put $\cH = H^2_d\big|_X$, $\cM = \mlt(\cH)$, $\cH_k = H^2_d\big|_{Y^{(k)}}$ and $\cM_k = \mlt(\cH_k)$. 
Then, the following statements are equivalent. 
\begin{enumerate}
\item $\tilde{\rho}_{s}(X,Y^{(k)}) \xrightarrow{k\to\infty} 0$. 
\item $\rho_{RK}(\cH,\cH_k) \xrightarrow{k\to\infty} 0$. 
\item $\rho_{M}(\cM, \cM_k)\xrightarrow{k\to\infty} 0$.  
\item $\inf\left\{\|\varphi\|_{cb} \|\varphi^{-1}\|_{cb} :\varphi : \cM \to \cM_k \textrm{ is an algebra isomorphism}\right\} \xrightarrow{k\to\infty} 1$.
\end{enumerate}
Moreover, if $(3)$ holds, that is, if $\rho_{M}(\cM, \cM_k)\xrightarrow{k\to\infty} 0$, then we have $\delta_{mult}(X,Y^{(k)}) \xrightarrow{k\to\infty} 1$.
\end{theorem}

\begin{proof}$(1)\implies (2)$ follows from Proposition \ref{prop:set_RKHS}, noting that after transforming $X$ so that one of its points is at the origin, the assumptions of that proposition hold uniformly for $X$ and all $Y^{(k)}$.
$(2)\implies (1)$ follows from Proposition \ref{prop:RKHS_set}.
$(2)\implies (3)$ follows from Proposition \ref{prop:BMspace_mult} and $(3)\implies (2)$ from \ref{prop:BMspace_mult_converse}. 
The equivalence $(3) \Longleftrightarrow (4)$ holds because the expression in $(4)$ is equal to $\delta_M(\cM,\cM_k)$ by the remarks following Definition \ref{def:BMmultdist} (see also the discussion in Section of \cite{Salomon-Shalit2016}). 
The final assertion follows from Proposition \ref{prop:mult_setmult} below. 
\end{proof}


\begin{proposition}\label{prop:mult_setmult}
Let $X = \{x_1, \ldots, x_n\}$ and $Y = \{y_1, \ldots, y_n\}$ be two subsets of the unit ball $\bB_d$. 
Then
\[
\delta_{mult}(X,Y) \leq \delta_{M}(\cH_1,\cH_2) .
\]
\end{proposition}
\begin{proof}
Let $\varphi: \cM_1 \to \cM_2$ be an isomorphism.
Let $z = (z_1, \ldots, z_d)$ denote the coordinate functions on $X_1$, which generate the algebra $\cM_1 = \mlt\cH_1$, where $\cH_1 = H^2_d\big|_{X_1}$. 
Consider $F = (F_1, \ldots, F_d) = \varphi(z) = (\varphi(z_1), \ldots, \varphi(z_d))$. 
As $z$ is a row contraction, we find that $\|F\|\leq \|\varphi\|_{cb}$. 
Since $F$ is the map corresponding to the adjoint $\varphi^* :\fM(\cM_2) \to \fM(\cM_1)$ associated to $\varphi$ as in Section \ref{sec:pre}, we have that $F : Y \to X$. 
In the same way we obtain a map $G : X \to Y$ associated with $\varphi^{-1}$ satisfying $\|G\| \leq \|\varphi^{-1}\|_{cb}$, and one readily shows that $G = F^{-1}$. 
Since an isomorphism is unit preserving, we have $\|\varphi\|, \|\varphi^{-1}\| \geq 1$, and so 
\[
\max\{\|F\|,\|G\|\}  \leq \|\varphi\|_{cb} \|\varphi^{-1}\|_{cb}.
\]
Taking the infimum over all isomorphisms $\varphi : \cM_1 \to \cM_2$ (which are all induced by composition with multipliers) we obtain $\delta_{mult}(X,Y) \leq \delta_{M}(\cH_1,\cH_2)$. 
\end{proof}

The proposition establishes the implication 
\[
\rho_{M}(\cM, \cM_k)\xrightarrow{k\to\infty} 0 \Longrightarrow \delta_{mult}(X,Y^{(k)}) \xrightarrow{k\to\infty} 1
\]
from Theorem \ref{thm:main}. 
Example \ref{ex:hartz} below shows that the converse does not hold.

\section{Examples}\label{sec:examples}

We close this paper by considering several examples highlighting what goes wrong when some of our assumptions are dropped. 

 The first two simple examples show that when we drop the assumption that the two varieties have the same finite cardinality, a small Hausdorff distance between varieties does not imply any kind of relationship between the RKHSs or the multiplier algebras. 

\begin{example}
Let $V = \{0\} \subset \bD$ and $W = \{0, \varepsilon\} \subset \bD$. 
One readily sees that $\rho_{H}(V,W) = \varepsilon$. 
On the other hand, $\cH_V$ and $\cM_V$ are one-dimensional, whereas $\cH_W$ and $\cM_W$ are two-dimensional, and no isomorphism exists between the spaces corresponding to the different varieties. 
\end{example}

The following is perhaps a more interesting example, in which the varieties have the same cardinality. 

\begin{example}
We define two sequences in the unit disc: $x_n = 1 - 2^{-n}$ and $y_n = 1 - (1 - \varepsilon^n) 2^{-n}$, where $\varepsilon \in (0,1)$. 
Clearly both sequences are Blaschke sequences, $\sum |1 - x_n|, \sum |1 - y_n| < \infty$, and so they are equal to the zero sets of bounded analytic functions on the disc. 
We can therefore define two multiplier varieties $V = \{x_n\}_{n=1}^\infty$ and $W = \{x_n\}_{n=1}^\infty \cup \{y_n\}_{n=1}^\infty$. 
Clearly $V$ and $W$ have the same cardinality. 
If $\varepsilon$ is sufficiently small, then the Hausdorff distance between $V$ and $W$ with respect to the either the hyperbolic metric or the Euclidean metric can be seen to be 
\[
\rho_{H}(V,W)  = \sup_n \left| \frac{x_n - y_n}{1 - x_n y_n}\right| \leq \varepsilon. 
\]
On the other hand, the condition of Hayman and Newman for a sequence of positive numbers to be interpolating (see p. 204 in \cite{Hoffman62}) shows that $V$ is an interpolating sequence while $W$ is not, so $\cM_V$ and $\cM_W$ cannot be isomorphic as algebras \cite[Theorem 6.3]{Davidson-Ramsey-Shalit-2015}. 
This means also that the RKHSs cannot be related by an isomorphism of RKHSs. 
\end{example}

In Theorem \ref{thm:main} the set $X$ and the spaces $\cH,\cM$ are held fixed. 
This is required in the proofs, because the bounds we prove relating the different distance functions depend on one of the sets. 
Our next example shows that this is really necessary.

\begin{example}\label{ex:hartz}
In this example we show that the map $X \mapsto \cH_X$ is not uniformly continuous on the space of two-point subsets in $\bD$ with diameter less than $1/2$. 

We consdier sets of the form $X = \{0,x\}$ and $Y = \{0,y\}$. 
Then $\tilde{\rho}_s(X,Y) \leq 2\max\{|x|,|y|\}$. 
We will show that $\tilde{\rho}_s(X,Y)$ can be arbitrarily small, while $\rho_{RK}(\cH_X,\cH_Y) \geq \log \alpha$ where $\alpha  := 1 + 10^{-2}$.
Suppose, with the intent of obtaining a contradiction, that this is not the case. 
Then there exists some $\delta >0$, such that $|x|,|y|<\delta$ implies that $\delta_{RK}(\cH_X,\cH_Y) < \alpha$. 
Fixing $|x|<\delta$, we find that for all $|y|<\delta$  there are $\lambda_1(y), \lambda_2(y) \in \bC$ such that 
\[
\begin{pmatrix}
|\lambda_1(y)|^2 & \lambda_1(y)\overline{\lambda_2(y)} \\
\lambda_2(y)\overline{\lambda_1(y)} & \frac{|\lambda_2(y)|^2}{1-|x|^2} 
\end{pmatrix} 
\leq 
\alpha \begin{pmatrix}
1 & 1 \\
1 & \frac{1}{1-|y|^2} 
\end{pmatrix}.
\]
As in the proof of Proposition \ref{prop:RKHS_set}, one can show that $\lambda_1(y)$ and $\lambda_2(y)$ stay in a neighborhood of $1$ (this is where we use that $\alpha$ is close to $1$), and therefore for every sequence of $y$'s tending to $0$, there is a subsequence such that the left hand side in the above inequality converges to a strictly positive matrix. 
On the other hand, the right hand side converges to a singular matrix as $y \to 0$, so in the limit we obtain an impossible inequality. 
\end{example}


After the first version of this paper appeared on the arxiv, Michael Hartz shared with us an example that shows that the implication 
\[
\delta_{mult}(X,Y^{(k)}) \xrightarrow{k\to\infty} 1 \Longrightarrow \rho_{M}(\cM_X, \cM_{Y^{(k)}})\xrightarrow{k\to\infty} 0
\]
does not hold, in general \cite{Hartz16p}. 
This led us to the following example and comments, and to several other improvements in the paper. 
We are grateful to Dr. Hartz for contacting us and permitting us to include his ideas in this paper. 
\begin{example}

Fix $r \in (0,1/6)$ and let $X = \{0,r\}$. 
Let $y_k = 1 - 2^{-k}$ and put $Y^{(k)} = \{y_k, y_{k+1}\}$. 
We first show that $\lim_{k\to\infty} \delta_{mult}(X,Y^{(k)}) = 1$. 

As $\rho_{\rm{ph}}\left(k/2^k,k/2^{k+1}\right) \xrightarrow{k\to\infty} 0 < r$, Pick's interpolation theorem for two points means that for all large enough $k$ there exists a function $h_k \in H^\infty$ that satisfies $h_k(0) = k/2^k$, $h_k(r) = k/2^{k+1}$, and $\|h_k\|_\infty \leq 1$. 
Letting $f_k = 1 - k^{-1}h_k$, we have $f_k(0) = y_k$, $f_k(r) = y_{k+1}$ and $\|f_k\|_\infty \xrightarrow{k\to \infty} 1$. 
On the other hand since $\rho_{\rm{ph}}(y_k, y_{k+1}) \xrightarrow{k \to \infty} 1/3 > r$, Pick's theorem implies that there is $g_k \in H^\infty$ such that $g_k(y_k) = 0$, $g_k(y_{k+1}) = r$, and $\|g_k\|_\infty \leq 1$. 
It follows that $\limsup_k \delta_{mult}(X,Y^{(k)}) \leq 1$. 
But for two-point sets we always have the inequality $\delta_{mult}(X,Y^{(k)}) \geq 1$, whence 
\[
\lim_{k\to\infty} \delta_{mult}(X,Y^{(k)}) = 1 .
\]
Since $\rho_{\rm{ph}}(y_k, y_{k+1}) \xrightarrow{k \to \infty} 1/3 > r = \rho_{\rm{ph}}(0,r)$, the convergnece $\tilde{\rho}_{s}(X,Y^{(k)}) \xrightarrow{k\to\infty} 0$ cannot hold.
Theorem \ref{thm:main} implies that $\rho_{M}(\cM_X, \cM_{Y^{(k)}})\xrightarrow{k\to\infty} 1$ also cannot hold. 

Alternatively, we can see that $\rho_{M}(\cM_X, \cM_{Y^{(k)}})\xrightarrow{k\to\infty} 1$ doesn't hold also as follows: 
suppose that $\varphi: \cM_{X} \to \cM_{Y^{(k)}}$ is an isomorphism that is induced by a bijection $G : Y^{(k)} \to X$. 
Then by the proof of Proposition 6.2 in \cite{Davidson-Hartz-Shalit2015}, 
\[
\|\varphi^{-1}\| \geq   \frac{\rho_{\rm{ph}}(y_k,y_{k+1})}{2\rho_{\rm{ph}}(G(y_k),G(y_{k+1}))}  =  \frac{\rho_{\rm{ph}}(y_k,y_{k+1})}{2\rho_{\rm{ph}}(0,r)} \xrightarrow{k \to \infty} \frac{1/3}{2r} . 
\]
This shows that $\liminf_k \rho_{M}(\cM_{X}, \cM_{Y^{(k)}}) \geq \frac{1}{6r} > 1$, proving what we claimed. 
\end{example}




\end{document}